\documentclass{amsart}

\usepackage{macros}
\standardsettings
\colorcommentstrue

\newcommand{\Bad}{\mathbf{Bad}}

\draftfalse

\begin{document}

\title{$\Bad(s,t)$ is hyperplane absolute winning}

\authorerez\authordavid

\begin{abstract}
J. An \cite{An} proved that for any $s,t \geq 0$ such that $s + t = 1$, $\Bad(s,t)$ is $(34\sqrt 2)^{-1}$-winning for Schmidt's game. We show that using the main lemma from \cite{An} one can derive a stronger result, namely that $\Bad(s,t)$ is hyperplane absolute winning in the sense of \cite{BFKRW}. As a consequence one can deduce the full dimension of $\Bad(s,t)$ intersected with certain fractals.
\end{abstract}
\maketitle

\section{Statement of results}

Throughout this paper, fix $s,t\geq 0$ with $s + t = 1$. Let $\Bad(s,t)$ denote the set
\[
\Bad(s,t) = \{(x,y) \in\R^2: \inf_{q\in\N}\max(q^s\|q x\|,q^t\|q y\|) > 0\},
\]
where $\|\cdot\|$ is the distance to the nearest integer. Schmidt's conjecture, proven in \cite{BPV}, states that $\Bad\left(\frac 13,\frac 23\right)\cap \Bad\left(\frac 23,\frac 13\right)$ is nonempty. A stronger result was proven by J. An \cite[Theorem 1.1]{An}, who showed that $\Bad(s,t)$ is $(34\sqrt 2)^{-1}$-winning for Schmidt's game. In particular this implies (cf. \cite[Theorem 2]{Schmidt1} and \cite[Corollary 2 of Theorem 6]{Schmidt1}) that for any countable collection of pairs $(s_n,t_n)_{n = 1}^\infty$, the intersection $\bigcap_{n\in\N}\Bad(s_n,t_n)$ is nonempty and in fact has full Hausdorff dimension in $\R^2$.

The object of this note is to give a proof of the following strengthening of An's theorem:

\begin{theorem}
\label{maintheorem}
The set $\Bad(s,t)$ is hyperplane absolute winning in the sense of \cite{BFKRW}.
\end{theorem}

Theorem \ref{maintheorem} is a generalization of An's theorem since every hyperplane absolute winning set is $\alpha$-winning for Schmidt's game for every $0 < \alpha < 1/2$ \cite[Proposition 2.3(a)]{BFKRW}. Moreover, the intersection of a hyperplane absolute winning set with a hyperplane diffuse set which is the support of an Ahlfors regular measure (see \cite{BFKRW} for the definitions) has full dimension with respect to that set \cite[Theorems 4.7 and 5.3]{BFKRW}. In particular hyperplane absolute winning sets have full dimension intersection with many well-known fractals such as the Sierpinski triangle and the von Koch snowflake curve. Finally, the class of hyperplane absolute winning sets is closed under countable intersections \cite[Proposition 2.3(b)]{BFKRW}, and invariant under $\CC^1$ diffeomorphisms \cite[Proposition 2.3(c)]{BFKRW}. As a result we have the following:

\begin{corollary}
For any hyperplane diffuse set $K\subseteq\R^2$ which is the support of an Ahlfors regular measure, for any countable collection of pairs $(s_n,t_n)_{n = 1}^\infty$, and for any countable collection of $\CC^1$ diffeomorphisms $(f_n)_{n = 1}^\infty$ from $\R^2$ to itself, the intersection
\[
K\cap \bigcap_{n\in\N}f_n\left(\Bad(s_n,t_n)\right)
\]
has full dimension in $K$.
\end{corollary}

We remark also that while the strategy for Schmidt's game given in An's paper is not explicit, depending on K\doubleacute onig's lemma (cf. \cite[Proposition 2.2]{An}), the strategy which we give in the proof of Theorem \ref{maintheorem} can be described more or less explicitly; see Remark \ref{remarkexplicitstrategy}.

\begin{remark}
The cases $s = 0$ and $t = 0$ of Theorem \ref{maintheorem} are trivial consequences of the fact that the set of badly approximable numbers is absolute winning (see \cite[Theorem 1.3]{McMullen_absolute_winning} or \cite[Theorem 2.5]{BFKRW}) and will be omitted. Throughout the proof we assume that $s,t > 0$.
\end{remark}

\begin{remark}
Although the higher-dimensional analogue of Schmidt's conjecture has been established by V. V. Beresnevich \cite{Beresnevich}, it is still not known, for example, whether $\Bad(s,t,u)$ is winning for all $s,t,u\geq 0$ with $s + t + u = 1$, where $\Bad(s,t,u)$ is defined appropriately.
\end{remark}

\textbf{Acknowledgements: }The ideas for this paper were developed during the special session on Diophantine approximation on manifolds and fractals at the AMS sectional meeting in the University of Colorado at Boulder, which took place in April 2013. The authors would like to thank the organizers for this very stimulating occasion. The first-named author would like to thank Barak Weiss for many inspiring and encouraging discussions about the ideas in this paper. Part of this work was supported by ERC starter grant DLGAPS 279893 and BSF grant 2010428.

\section{Preliminaries}

The proof of Theorem \ref{maintheorem} will consist of combining the main idea of \cite{An} with the main idea of \cite[Appendix C]{FSU4}. We therefore begin by recalling these ideas.

\subsection{The main lemma of \cite{An}}
For each $P=\left(\frac{p}{q},\frac{r}{q}\right)\in\Q^2$ and $\epsilon > 0$, following \cite{An} we let\footnote{We remark that the $c$ in \cite{An} corresponds to our $\epsilon$; our $c$ corresponds to the $c$ in \cite[Appendix C]{FSU4}.\label{F1}}
\[
\Delta_\epsilon(P) = \left\{(x,y)\in\R^2: \left|x - \frac pq\right| \leq \frac{\epsilon}{q^{1 + s}} \text{ and } \left|y - \frac rq\right| \leq \frac{\epsilon}{q^{1 + t}} \right\},
\]
so that
\begin{equation}
\label{Badstformula}
\Bad(s,t) = \bigcup_{\epsilon > 0} \left(\R^2\setminus\bigcup_{P\in\mathbb{Q^2}}\Delta_\epsilon(P)\right).
\end{equation}
Let $\scrL$ denote the collection of lines (hyperplanes) in $\R^2$. If $L\in\scrL$ and $\gamma > 0$, we let $L^{(\gamma)}$ denote the \emph{$\gamma$-thickening} of $L$, i.e. the set $\{(x,y)\in\R^2: \dist((x,y),L)\leq\gamma\}$.

\begin{lemma}[{\cite[Lemma 4.2]{An}}]
\label{lemmaan}
Fix $R > 1$ and $\ell > 0$. There exists $\epsilon > 0$ and a partition
\begin{equation}
\label{partition}
\Q^2=\bigcup_{\delta = 1}^2 \bigcup_{n = 1}^\infty \bigcup_{k = 1}^n \PP_{n,k}^{(\delta)}
\end{equation}
such that the following holds: For each $m\geq 0$, let
\[
\PP_m = \bigcup_{\delta = 1}^2\bigcup_{k = 1}^m\PP_{m,k}^{(\delta)},
\]
and let $\BB_m$ denote the collection of balls $B\subset\R^2$ of radius $R^{-m}\ell/2$ satisfying
\begin{equation}
\label{Bhypothesis}
\forall m' \leq m \all P\in\PP_{m'}, \;\;\Delta_\epsilon(P)\cap B = \emptyset.
\end{equation}
Then for all $n\geq k\geq 1$, for all $\delta\in\{1,2\}$, and for all $B\in\BB_{n - k}$, there is a line $L = L_{n,k,\delta}(B)\in\scrL$ such that
\begin{equation}
\label{an}
\Delta_\epsilon(P) \subset L^{\left(\frac 13\ell R^{-n}\right)} \all P\in\PP_{n,k}^{(\delta)} \text{ such that } \Delta_\epsilon(P)\cap B\neq\emptyset.
\end{equation}
\end{lemma}

\begin{remark}
The relation between Lemma \ref{lemmaan} and \cite{An} requires some explanation. First of all, given $R > 1$ and $\ell > 0$, one can let $\epsilon > 0$ be defined by the equation \cite[(3.2)]{An}. Next, one can define the partition \eqref{partition} as in \cite[p. 5-6]{An}. At this point \cite[Lemma 4.2]{An} can almost be read as stated, except that An has fixed $\tau\in\SS_{n - k}$ instead of $B\in\BB_{n - k}$, and has considered the set $\PP_{n,k}^{(\delta)}(\tau) = \{P\in\PP_{n,k}^{(\delta)}:\Phi(\tau)\cap\Delta_\epsilon(P)\neq\emptyset\}$ in place of the set $\{P\in\PP_{n,k}^{(\delta)} : \Delta_\epsilon(P)\cap B\neq\emptyset\}$. But we observe that for $\tau\in\TT_{n - k}$, we have $\tau\in\SS_{n - k}$ if and only if $\Phi(\tau)\in\BB_{n - k}$.\footnote{Here we ignore the distinction between balls and squares. The difference is important only in calculating diameter; the diameter of a square with respect to the max norm is equal to its side length, while the diameter of a ball is equal to twice its radius. This is why we require balls in $\BB_m$ to have radius $R^{-m}\ell/2$, while if $\tau\in\SS_m$, the side length of $\Phi(\tau)$ is $R^{-m}\ell$.} Moreover, the proof of \cite[Lemma 4.2]{An} works equally well if $\Phi(\tau)$ is replaced by an arbitrary element in $\BB_{n - k}$. Thus the lemma holds just as well if $\Phi(\tau)$ denotes an arbitrary element of $\BB_{n - k}$ rather than an arbitrary element of $\Phi(\SS_{n - k})$.
\end{remark}

\subsection{Two variants of Schmidt's game}
We proceed to describe two variants of Schmidt's game, one introduced in \cite{BFKRW} and the other introduced in \cite[Appendix C]{FSU4}. In this paper we will not deal directly with the first game, but we will prove that $\Bad(s,t)$ is winning with respect to the second game. Since the two games are equivalent (Lemma \ref{lemmaFSU} below), this proves that $\Bad(s,t)$ is also winning with respect to the first game, and therefore has the large dimension properties described in the introduction.

\begin{definition}
\label{definitionHPgame}
Given $0 < \beta < 1/3$, Alice and Bob play the \emph{$\beta$-hyperplane absolute game} as follows:
\begin{itemize}
\item[1.] Bob begins by choosing a ball $B_0 = B(z_0,r_0)\subset \R^2$.
\item[2.] On Alice's $n$th turn, she chooses a set of the form $L_n^{(\w r_n)}$ with $L_n\in\scrL$, $0 < \w r_n \leq \beta r_n$, where $r_n$ is the radius of Bob's $n$th move $B_n = B(z_n,r_n)$. We say that Alice \emph{deletes} her choice $L_n^{(\w r_n)}$.
\item[3.] On Bob's $(n + 1)$st turn, he chooses a ball $B_{n + 1} = B(z_{n + 1},r_{n + 1})$ satisfying
\begin{equation}
\label{bobrulesabsolute}
r_{n + 1}\geq\beta r_n \text{ and } B_{n + 1} \subset B_n \butnot L_n^{(\w r_n)},
\end{equation}
where $B_n = B(z_n,r_n)$ was his $n$th move, and $L_n^{(\w r_n)}$ was Alice's $n$th move.
\item[4.] If $r_n\not\to 0$, then Alice wins by default. Otherwise, the balls $(B_n)_1^\infty$ intersect at a unique point which we call the \emph{outcome} of the game.
\end{itemize}
If Alice has a strategy guaranteeing that the outcome lies in a set $S$ (or that she wins by default), then the set $S$ is called \emph{$\beta$-hyperplane absolute winning}. If a set $S$ is $\beta$-hyperplane absolute winning for all $0 < \beta < 1/3$ then it is called \emph{hyperplane absolute winning}.

By contrast, given $\beta,c > 0$, Alice and Bob play the \emph{$(\beta,c)$-hyperplane potential game} as follows:
\begin{itemize}
\item[1.] Bob begins by choosing a ball $B(\xx_0,r_0)\subset \R^2$.
\item[2.] For each $n$, after Bob makes his $n$th move $B_n = B(\xx_n,r_n)$, Alice will make her $n$th move. She does this by choosing a countable collection of sets of the form $L_{i,n}^{(r_{i,n})}$, with $L_{i,n}\in\scrL$ and $r_{i,n} > 0$, satisfying
\begin{equation}
\label{alicerules}
\sum_i r_{i,n}^c \leq (\beta r_n)^c.
\end{equation}
\item[3.] After Alice makes her $n$th move, Bob will make his $(n + 1)$st move by choosing a ball $B_{n + 1} = B(\xx_{n + 1},r_{n + 1})$ satisfying
\begin{equation}
\label{bobrules}
r_{n + 1}\geq\beta r_n \text{ and } B_{n + 1} \subset B_n,
\end{equation}
where $B_n = B(\xx_n,r_n)$ was his $n$th move.
\item[4.] If $r_n\not\to 0$, then Alice wins by default. Otherwise, the balls $(B_n)_1^\infty$ intersect at a unique point which we call the \emph{outcome} of the game. If the outcome is an element of any of the sets $L_{i,n}^{(r_{i,n})}$ which Alice chose during the course of the game, she wins by default.
\end{itemize}
If Alice has a strategy guaranteeing that the outcome lies in a set $S$ (or that she wins by default), then the set $S$ is called \emph{$(\beta,c)$-hyperplane potential winning}. If a set is $(\beta,c)$-hyperplane potential winning for all $\beta,c > 0$, then it is \emph{hyperplane potential winning}.
\end{definition}

The following lemma is a special case of the main result of \cite[Appendix C]{FSU4}:

\begin{lemma}[{\cite[Theorem C.8]{FSU4}}]
\label{lemmaFSU}
A set is hyperplane potential winning if and only if it is hyperplane absolute winning.
\end{lemma}

\begin{proof}[Sketch of the proof]
We sketch only the forward direction, as it is the one which we use. Suppose that $S\subset\R^d$ is hyperplane potential winning. Let $\beta > 0$. Fix $\w\beta,c > 0$ small to be determined, and consider a strategy of Alice which is winning for the $(\w\beta,c)$-hyperplane potential game. Each time Bob makes a move $B_n = B(z_n,r_n)$, Alice chooses a collection of sets $\{L_{i,n}^{(r_{i,n})}\}_{i = 1}^{N_n}$ (with $N_n\in\N\cup\{\infty\}$) satisfying \eqref{alicerules}. Alice's corresponding strategy in the $\beta$-hyperplane absolute game will be to choose her set $L_n^{(\beta r_n)}\subset \R^2$ so as to maximize
\begin{equation}
\label{alicemaximizes}
\phi(B_n;L_n^{(\beta r_n)}) := \sum_{m = 0}^n \sum_{\substack{i \\ L_{i,m}^{(r_{i,m})}\subset L_n^{(\beta r_n)} \\ L_{i,m}^{(r_{i,m})}\cap B_n\neq\smallemptyset}} r_{i,m}^c.
\end{equation}
Suppose that Alice plays according to this strategy, and let $(B_n)_1^\infty$ be the sequence of Bob's moves. For each $n\in\N$ let
\[
\phi(B_n) = \sum_{m = 0}^n \sum_{\substack{i \\ L_{i,m}^{(r_{i,m})}\cap B_n\neq\smallemptyset}} r_{i,m}^c.
\]
One demonstrates by induction on $n$ (see \cite[Appendix C]{FSU4} for details) that if $\w\beta$ and $c$ are chosen sufficiently small, then there exists $\epsilon > 0$ such that for all $n\in\N$,
\begin{equation}
\label{inductivehypothesis}
\phi(B_n) \leq (\epsilon r_n)^c.
\end{equation}
Intuitively, the reason for this is that Alice is ``deleting the regions with high $\phi$-value'', and is therefore minimizing the $\phi$-value of Bob's balls. Thus she is forcing $\phi(B_n)$ to be as small as possible.

Now suppose that $r_n\to 0$; otherwise Alice wins the $\beta$-hyperplane absolute game by default. Then \eqref{inductivehypothesis} implies that $\phi(B_n)\to 0$. In particular, for each $(i,m)$, $r_{i,m}^c > \phi(B_n)$ for all $n$ sufficiently large which implies $L_{i,m}^{(r_{i,m})}\cap B_n = \emptyset$. Thus the outcome of the game does not lie in $L_{i,m}^{(r_{i,m})}$ for any $(i,m)$, so Alice does not win the $(\w\beta,c)$-hyperplane potential game by default. Thus if she wins, then she must win by having the outcome lie in $S$. Since the outcome is the same for the $(\w\beta,c)$-hyperplane potential game and the $\beta$-hyperplane absolute game, this implies that she also wins the $\beta$-hyperplane absolute game.
\end{proof}

\begin{remark}
\label{remarkexplicitstrategy}
The proof of Lemma \ref{lemmaFSU} shows that if Alice has an explicit strategy to win the hyperplane potential game, then she also has an explicit strategy to win the hyperplane absolute game. Here by ``explicit'' we mean roughly that Alice can calculate what her next move should be in a finite amount of time, based on the input of Bob's previous move. To see that Alice's strategy in the above proof can be made explicit, note that Alice does not actually have to maximize \eqref{alicemaximizes} precisely; for example, if she instead makes a choice for which the value of \eqref{alicemaximizes} is at least $1/2$ of the maximum value, then the proof will still go through with minor modifications. This margin of error allows her to restrict her attention to a finite set of potential moves; moreover, for each potential move she can truncate the series \eqref{alicemaximizes} at an appropriate point to make it a finite series. Thus she can compute an ``approximate maximum'' explicitly.
\end{remark}

\section{Proof of Theorem \ref{maintheorem}}

In this section we prove Theorem \ref{maintheorem}, which states that $\Bad(s,t)$ is hyperplane absolute winning.

\begin{proof}[Proof of Theorem \ref{maintheorem}]
By Lemma \ref{lemmaFSU}, it suffices to show that $\Bad(s,t)$ is hyperplane potential winning. Fix $\beta,c > 0$, and we will show that it is $(\beta,c)$-hyperplane potential winning. Let $B(\xx_0,r_0)\subset\R^2$ be Bob's first move. Fix $R \geq \beta^{-1}$ to be determined (depending on $\beta$ and $c$), and let $\ell = 2r_0$. Let $\epsilon > 0$ and the partition \eqref{partition} be as in Lemma \ref{lemmaan}. Alice's strategy will be defined by infinitely many ``triggers'' as follows: For each $m\geq 0$, Alice will wait until Bob chooses a ball $B_j = B(\xx_j,r_j)$ that satisfies $r_j \leq R^{-m}r_0/2$. The first $j$ for which this inequality holds will be denoted $j_m$, with $j_m = \infty$ if it never holds. We observe that
\begin{itemize}
\item[(i)] $j_m \geq 1$ for all $m\geq 0$, since $r_0 > R^{-m}r_0/2$, and
\item[(ii)] $j_{m + 1} \geq j_m + 1$, since $r_{j_m} \geq \beta r_{j_m - 1} > \beta R^{-m}r_0/2 \geq R^{-(m + 1)}r_0/2$ (using the fact that $R\geq \beta^{-1}$).
\end{itemize}
Fix $m\geq 0$, and let $j = j_m$. Let $\w B_j = B(\xx_j,R^{-m}r_0)$. On Alice's $j$th turn, her strategy will be as follows:
\begin{itemize}
\item[(1)] If $\w B_j\notin\BB_m$, then she will do nothing.
\item[(2)] If $\w B_j\in\BB_m$, then for each $k\geq 1$ and $\delta\in\{1,2\}$ she will apply Lemma \ref{lemmaan} to get a line $L_{k,\delta} = L_{m + k,k,\delta}(\w B_j) \in \scrL$, and she will delete the hyperplane-neighborhood $L_{k,\delta}^{(3R^{-(m + k)}r_0)}$.
\end{itemize}
The legality of this action is guaranteed by (ii), which shows that Alice is not deleting multiple collections on the same turn, together with the inequality
\begin{align*}
\sum_{\delta = 1}^2 \sum_{k = 1}^\infty \left(3R^{-(m + k)}r_0\right)^c
&= (3 R^{-m}r_0)^c \; 2\sum_{k = 1}^\infty R^{-ck}\\
&\leq \left(\beta^2 R^{-m}r_0/2\right)^c &\text{(for $R$ chosen large enough)}\\
&< \left(\beta r_{j_m}\right)^c. \since{$r_{j_m} \geq \beta r_{j_m - 1} > \beta R^{-m}r_0/2$}
\end{align*}
To complete the proof, we must show that the strategy described above guarantees a win for Alice. By contradiction, suppose that Bob can play in a way so that Alice loses. By definition, this means that the radii of Bob's balls tend to zero, and that their intersection point $\xx\in\R^2$ is not in $\Bad(s,t)$ nor in any of the hyperplane-neighborhoods which Alice deleted in the course of the game. In particular, the radii tending to zero means that each of the triggers happens eventually, i.e. $j_m < \infty$ for all $m\geq 0$.

\begin{claim}
\label{claiminductive}
For all $m\geq 0$, $\w B_{j_m} := B(\xx_{j_m},R^{-m}r_0)\in \BB_m$.
\end{claim}
\begin{subproof}
We proceed by strong induction and contradiction. Suppose the claim holds for all $0\leq m < M$, but does not hold for $M$. Then there exist $M'\leq M$ and $P\in\PP_{M'}$ such that $\Delta_\epsilon(P)\cap \w B_J \neq\emptyset$, where $J = j_M$. Write $P\in\PP_{M',k}^{(\delta)}$ for some $1\leq k\leq M'$ and $\delta\in\{1,2\}$. Let $m = M' - k < M$, and let $j = j_m$. We apply the induction hypothesis to see that $\w B_j\in \BB_m$. Thus on Alice's $j$th turn, she must have deleted the hyperplane-neighborhood $A := L_{k,\delta}^{(3R^{-(m + k)}r_0)}$, where $L_{k,\delta} = L_{m + k,k,\delta}(\w B_j)$ is as in Lemma \ref{lemmaan}.

On the other hand, since $J\geq j$, we have $B_J \subset B_j$; thus $\dist(\xx_j,\xx_J) \leq r_j \leq R^{-m}r_0/2$ and so $\w B_J \subset \w B_j$. Combining this with the contradiction hypothesis gives $\Delta_\epsilon(P)\cap \w B_j\neq\emptyset$. So by the definition of $L_{k,\delta} = L_{m + k,k,\delta}(\w B_j)$ (cf. \eqref{an}), we have
\[
\Delta_\epsilon(P) \subset L_{k,\delta}^{\left(\frac 23 R^{-(m + k)}r_0\right)}.
\]
Since $\frac 23 R^{-(m + k)}r_0 + 2 R^{-M}r_0 \leq 3 R^{-(m + k)}r_0$, this implies $(\Delta_\epsilon(P))^{(2 R^{-M}r_0)} \subset A$. In particular, since $\Delta_\epsilon(P)\cap \w B_J \neq\emptyset$ we have
\[
\xx\in B_J \subset \w B_J \subset (\Delta_\epsilon(P))^{\left(2 R^{-M}r_0\right)} \subset A.
\]
This demonstrates that Alice won by default, contradicting our hypothesis.
\end{subproof}

Now for all $P\in\Q^2$, we have $P\in\PP_m$ for some $m\geq 1$. Let $j = j_m$. Applying Claim \ref{claiminductive}, we see that $\Delta_\epsilon(P)\cap \w B_j = \emptyset$. But $\xx\in B_j \subset \w B_j$, so $\xx\notin \Delta_\epsilon(P)$. By \eqref{Badstformula}, this means $\xx\in\Bad(s,t)$. So Alice won, contradicting our hypothesis. This completes the proof.
\end{proof}

\ignore{
To devise a strategy for Alice we have to guarantee condition (\ref{eq:Jinpeng2D}), in which $\delta$ do not appear. This is why Alice will play simultanuously with $\delta$ equals $1$ and $2$. Let Bob choose a ball $B=B\left(\xx_{1},r_{1}\right)$ and a $\beta>0$. Use Lemma \ref{lemmaan} with $R=\frac{1}{\beta^2}$ and $r=r_{1}$ to get $\epsilon > 0$ and the partition $\PP$. Define Alice's strategy by induction. Assume Bob chose a ball $B_{2n-1}$ such that \eqref{eq:Jinpeng2D} holds with $k = 1$. Then for every $k$, let Alice choose
\[
L_{2n-1,k}=\begin{cases}
L_{n,\frac{k}{2},1}^{\left(R^{-\left(n-1+\frac{k}{2}\right)}r\right)} & k\textrm{ is even }\\
\emptyset & \textrm{otherwise}
\end{cases}
\]
where $L_{n,k,1}$ is the unique line in $L(\PP(n-1+k,k,1,B_{2n-1}))$, and (no matter what was Bob's choice for $B_{\ensuremath{2n}}$)
\[
L_{2n,k}=\begin{cases}
\emptyset & k\textrm{ is even }\\
L_{n,\frac{k+1}{2},2}^{\left(R^{-\left(n+\frac{k-1}{2}\right)}r\right)} & \textrm{otherwise}
\end{cases}
\]
where $L_{n,k,2}$ is the unique line in $L(\PP(n-1+k,k,2,B_{2n-1}))$. To justify the induction hypothesis, we have to check that the next move of Bob, $B_{2n+1}$, has to satisfy $\Delta_\epsilon(P)\cap B_{2n+1}=\emptyset$ for any $P\in\PP(n_0)$ with $n_0\leq n$. Since $B_{2n+1}\subset B_{2n-1}$ it is enough to check it for $n_0=n$. Now, for any $P\in\PP(n)$ that satisfies $\Delta_\epsilon(P)\cap B_{2n}\neq\emptyset$ there exist $\delta\in\left\{ 1,2\right\} $ and $k\geq1$, such that $P\in\PP(n,k,\delta,B_{2n})$, therefore by Lemma \ref{lemmaan} it satisfies $\Delta_\epsilon(P)\subset L(P)^{\left(R^{-n}r\right)}=L_{n+1-k,k,\delta}^{\left(R^{-n}r\right)}$. Indeed we have by definition of the game,
\[
B_{2n+1}\subset B_{2n}\setminus\bigcup_{k=1}^{2n}L_{2n+1-k,k}=B_{2n}\setminus\bigcup_{\delta=1}^2\bigcup_{k=1}^{n}L_{n+1-k,k,\delta}^{\left(R^{-n}r\right)}
\]
}

\ignore{

The reader is advised to look in \cite[Appendix C]{FSU4} for a more general version of this game which is called the potential game. Actually, it is proved there that being a winning set for the potential game with respect to hyperplanes is equivalent to being HAW. In order to be self contained, we repeat the argument of the proof that appears there for the special case we are interested in this paper.

\begin{theorem}
A winning set for the game described in Definition \ref{definitionHPgame} is HAW on $\R^2$.\end{theorem}
\begin{proof}
TBD
\end{proof}
}

\end{document}